\documentclass[12pt,thmsa, leqno]{amsart}

\let\oldtocsection=\tocsection

\let\oldtocsubsection=\tocsubsection

\renewcommand{\tocsection}[2]{\hspace{0em}\oldtocsection{#1}{#2}}
\renewcommand{\tocsubsection}[2]{\hspace{2em}\oldtocsubsection{#1}{#2}}

\usepackage{setspace}

\usepackage{amssymb, euscript,  mathrsfs}
\usepackage[dvips]{graphics}
\usepackage[title]{appendix}

\input{amssym.def}

\input{amssym.tex}

\usepackage{lineno}

\usepackage{tikz}
\usepackage[utf8]{inputenc}

\usetikzlibrary{matrix,arrows,decorations.pathmorphing}

\usepackage[all,cmtip]{xy}

\usepackage{amsmath}
\let\oldAA\AA
\renewcommand{\AA}{\text{\normalfont\oldAA}}

\def\rn{\bbr^n}
\def\rnk{\bbr^{n-k}}
\def\cgnk {{\rm Gr} (n,k)}
\def\sn{S^{n-1}}
\def\snk{S^{n-k-1}}

\def\lang{\langle}
\def\rang{\rangle}

\def\part{\partial}
\def\intl{\int\limits}

\def\t{\tau}

\def\e{\varepsilon}

\def\Gam{\Gamma}

\def\a{\alpha}

\def\vp{\varphi}

\def\gam{\gamma}

\def\supp{{\hbox{\rm supp}}}

\def\const{{\hbox{\rm const}}}

\def\W{\mathcal{W}}

\def\bbr{{\Bbb R}}

\newtheorem{theorem}{Theorem}[section]
\newtheorem{lemma}[theorem]{Lemma}

\theoremstyle{definition}

\theoremstyle{remark}
\newtheorem{remark}[theorem]{Remark}

\numberwithin{equation}{section}

\theoremstyle{corollary}
\newtheorem{corollary}[theorem]{Corollary}
\newtheorem{proposition}[theorem]{Proposition}

\newtheorem{question}[theorem]{Question}

\numberwithin{equation}{section}

\newcommand{\be}{\begin{equation}}
\newcommand{\ee}{\end{equation}}

\newcommand{\bea}{\begin{eqnarray}}
\newcommand{\eea}{\end{eqnarray}}
\newcommand{\Bea}{\begin{eqnarray*}}
\newcommand{\Eea}{\end{eqnarray*}}

\def\sideremark#1{\ifvmode\leavevmode\fi\vadjust{\vbox to0pt{\vss
 \hbox to 0pt{\hskip\hsize\hskip1em
\vbox{\hsize2cm\tiny\raggedright\pretolerance10000
 \noindent #1\hfill}\hss}\vbox to8pt{\vfil}\vss}}}%

\begin{document}


\title[The Inverse Problem]
{The Inverse Problem for the  Euler-Poisson-Darboux Equation and   Shifted $k$-Plane  Transforms }

\author{ B. Rubin}

\address{Department of Mathematics, Louisiana State University, Baton Rouge,
Louisiana 70803, USA}
\email{borisr@lsu.edu}





\begin{abstract}
The  inverse problem for the  Euler-Poisson-Darboux equation deals with reconstruction of the Cauchy data
 for this equation from incomplete information about its solution. In the present article, this problem is studied
 in connection with the injectivity of the shifted $k$-plane transform, which assigns to functions in  $L^p(\mathbb {R}^n)$  their mean values   over all
 k-planes  at a fixed  distance  from the given $k$-planes. Several generalizations, including the Radon transform  over strips of fixed width in $\mathbb {R}^2$
  and a similar transform over tubes of fixed diameter in $\mathbb {R}^3$,  are considered.
  \end{abstract}

\maketitle

\section{Introduction}

The present article grew up from the following problems in PDE and integral geometry, dealing with  generalizations of  the spherical means  and Radon transforms in $\rn$.

 \vskip 0.2 truecm

 \noindent {\bf Problem 1.}   We recall that a singular Cauchy problem for the
Euler-Poisson-Darboux equation (abbreviated as the EPD equation) is to find a function $u(x,t)$ on $\rn \times (0,\infty)$, $n\ge 1$, satisfying
\begin{equation}
\Delta_x u-u_{tt}-\frac{n+2\alpha-1}{t}\,u_{t}%
=0,\label{EPD.1}%
\end{equation}%
\begin{equation}
\lim\limits_{t \to 0} u(x,t)=f(x),\qquad \lim\limits_{t \to 0} u_{t}(x,t)=0,\label{EPD.2}%
\end{equation}
where $\Delta_x$ is the Laplace operator in the $x$-variable.  The case  $\alpha= (1-n)/2\,$ in (\ref{EPD.1}) gives the usual wave equation.
 If  $\alpha\ge (1-n)/2$,  then (\ref {EPD.1})-(\ref {EPD.2})
has a unique
solution \be\label{ijt} u(x,t)=(M_t^{\a}f)(x),\ee where  $M_t^{\a}f$ is defined
as  analytic continuation in the $\a$-variable of the integral
\be\label {hen}
(M_t^{\a}f)(x)  \!=\!\frac{\Gamma\left(  \alpha\!+\!n/2\right)  }{\pi^{n/2}\Gamma\left(
\alpha\right)  }\intl_{|y|<1}\!\!(1\!-\!|y|^{2})^{\alpha-1}%
f(x\!-\!ty)\,dy.\ee
 If  $\alpha<(1-n)/2$,   the unique
solution of (\ref {EPD.1})-(\ref {EPD.2})
is not possible;   see \cite{B2}  for details.
If  $\alpha=0$, then   $M_t^{0}f\equiv \lim\limits_{\a \to 0} M_t^{\a}f$  is the spherical mean
\be\label {mf}
(M_t^{0}f)(x)= \intl_{S^{n-1}} f (x-t\theta)\, d_*\theta, \ee
where  $S^{n-1}$ is the unit
sphere in $\rn$ and $d_*\theta$ stands for the normalized surface area measure
on $S^{n-1}$.
 In the case  $\alpha=1$, we have
 \be\label {hena}
(M_t^{1}f)(x)  \!=\!\frac{\Gamma\left(1+\!n/2\right)  }{\pi^{n/2}}\intl_{|y|<1}\!\!
f(x\!-\!ty)\,dy,\ee
which is the mean value of $f$ over the Euclidean ball with center $x$ and radius $t$.

\begin{question}  Is it possible to reconstruct the initial function $f$ from  the solution $u(x,t)$ if the latter is known only for  some fixed value $t\!=\!\rho>0$ ?
\end{question}

In the cases $\alpha=0$ and $\alpha=1$, when $u(x,t)$ is expressed by the mean value operators (\ref{mf}) and (\ref{hena}), this question is the well known Pompeiu problem in integral geometry for spheres and balls; see, e.g., \cite{RS, T, Vo, Z80}.

 \vskip 0.2 truecm

\noindent {\bf Problem 2.} Let $X$ be an $n$-dimensional constant curvature space,  and   $\Xi$ be the set of all $k$-dimensional totally geodesic submanifolds of $X$, $1\le k \le n-1$,  \cite{H11}.
Consider  the Radon type transform
\be\label{Radon}
(R_\rho f)(\xi)=\!\!\intl_{d(x,\xi)=\rho}\!\! f(x) dm(x), \qquad x\in X, \quad \xi \in \Xi, \quad \rho>0,\ee
where  $d (\cdot, \cdot)$ stands for the  geodesic distance on $X$ and  $dm(x)$ is the relevant canonical measure. Following  F.  Rouvi\`{e}re \cite [p. 19]{Rou}, we call $R_\rho f$ the   {\it shifted Radon transform}.
 The limiting case $\rho=0$ gives the usual totally geodesic Radon transform  \cite{H11}.

\begin{question}  Suppose that $\rho>0$ is fixed. Is it possible to reconstruct $f$ if $ (R_\rho f)(\xi)$ is known for all (or almost all) $\xi \in \Xi$?
\end{question}

This problem was studied in  \cite{Ru22} for the case when $X$ is the unit sphere in  $\bbr^{n+1}$. In the present paper, we assume that $X=\rn$ and $\Xi$ is the Grassmann manifold of  $k$-dimensional affine planes in $\rn$. We  also consider the accompanying operator
 \be\label{Rafe}
(\tilde R_\rho f)(\xi)=\!\!\intl_{d(x,\xi)< \rho}\!\! f(x)\, dx \ee
and ask the same question. Note that if we formally set $k=0$ in (\ref{Radon}) and (\ref{Rafe}), we arrive at the mean value operators  (\ref{mf}) and (\ref{hena}).
 Thus Problems 1 and 2 have  common nature.

Problem 2  can be traced back  to the celebrated  1917  paper by J. Radon \cite {Rad}, who investigated reconstruction of a function on the $2$-plane from the integrals of this function over straight lines.

\begin{question} We wonder, whether   the straight lines in  Radon's  problem  can be  replaced by  the strips of fixed width? Can we treat a similar problem  for pipes or solid tubes of fixed diameter in $\bbr^3$?
 \end{question}

These questions, which may look naive at first glance, will be studied below from the point of view of injectivity of our operators on $L^p$ functions
in the  framework of the more  general analytic family of operators  $R^{\a}_\rho$ (see Section \ref {polk}), including  (\ref{Radon}) and (\ref{Rafe})  as particular cases.

In Section 2 we formulate the main results (see Theorems \ref {laaTi}, \ref{lzzt}, and \ref{lzzt2}). Section 3 contains necessary preliminaries. In Section 4 we prove Theorem \ref {laaTi}. Theorems \ref{lzzt} and \ref{lzzt2} are proved in Section 5.

\section{Main Results.}
In the following, we use the standard notation $J_\nu (r)$ for the Bessel function of the first kind \cite{Er}. 

\subsection {The EPD equation} The operator $M_t^{\a}$ is a convolution with the compactly supported tempered distribution
 \be\label {mmb}
 m_t^{\a} (y)= \frac{\Gamma\left(  \alpha\!+\!n/2\right)  }{\pi^{n/2}\Gamma\left(
\alpha\right)  } \, t^{-n} (1\!-\!|y/t|^{2})_+^{\alpha-1},\ee
 defined as analytic continuation of the corresponding expression with $Re \,\a >0$.
If $\a \ge 0$, then $M_t^{\a}$
 is obviously bounded on $L^p (\rn)$ for all $1\le p \le \infty$. However,
the oscillatory behaviour of its Fourier multiplier $ \hat m_t^{\a} (\xi)$ enables one to treat the case  $\a <0$ as well. Different special cases were studied
by    Miyachi  \cite{Mi},  Peral  \cite{Per},   Stein  \cite{St76},  Strichartz \cite{Str70}, to mention a few.
 In particular,  if $1<p<\infty$, then $M_t^{\a}$ extends as a linear bounded operator from $L^p (\rn)$ to itself if and only if
\be\label {poi}(n-1)|1/p -1/2| \le \a + (n-1)/2.\ee
More general statements, including action  from $L^p (\rn)$ to $L^q (\rn)$ with different $p$ and $q$,  and further references
can be found in  \cite {Ru87, Ru89}.

Our first result is the following theorem, in which $\alpha\ge (1-n)/2$ and $\rho$ is a fixed positive number.

\begin{theorem}\label {laaTi} ${}$\hfill

\vskip 0.2 truecm

\noindent {\rm (i)} The  initial function $f\in L^p (\rn)$ in  (\ref {EPD.1})-(\ref {EPD.2}) can be
uniquely reconstructed from  $u(x,t)\vert_{t=\rho}$, provided that $1\le p\le  2n/(n-1)$ if $n>1$ or $p<\infty$ if $n=1$.

\vskip 0.2 truecm

\noindent {\rm (ii)}  Let  $p> 2n/(n-1)$ if $n>1$, or $p=\infty$ if $n=1$. If  $J_{n/2 +\a -1} (\rho)=0$,   then
 $f\in L^p (\rn)$ cannot be  uniquely reconstructed if the solution  $u(x,t)$  is known only for $t=\rho$.

\vskip 0.2 truecm

\noindent {\rm (iii)}
For all $1\le p \le \infty$ and $n \ge 1$, the  function $f\in L^p (\rn)$
 is uniquely defined if the solution  $u(x,t)$ is known for any  two  fixed values $t=\rho_1$ and
 $t=\rho_2$, provided that $\rho_1/\rho_2$ is not a quotient of zeros of the Bessel function  $J_{n/2 +\a -1}$.

\end{theorem}

Some parts of this theorem are known in a different form. Specifically, if $n\ge 2$, then  {\rm (i)}
 follows from the more general  result of Rawat and  Sitaram \cite[Proposition 3.1]{RS} for arbitrary compactly supported radial distributions; see  also
  Agranovsky and  Narayanan \cite [Theorem 4]{AN} and   Thangavelu \cite [Theorem 2.2] {T} for  the case $\a=0$.
 The statement (ii) is
 known  for $\a=0$; cf. \cite [formula (2.6)]{T}, \cite  [formula (1.1)]{AN}.
The statement (iii) mimics the two-radius theorem from  \cite [Theorem 2.1] {T}, corresponding to $\a=0$. See Remark  \ref{kasz}  for further comments.

\subsection{The operators $R^{\a}_\rho$}\label {polk}   Let $\cgnk$  be the Grassmann manifold  of  non-oriented $k$-dimensional
affine planes in $\bbr^n$, $1\le k\le n-1$, and let $\rho$ be a fixed positive integer.  The operators (\ref{Radon})
and (\ref{Rafe}) can be formally written as
\be\label{Radonz}
(R_\rho f)(\t)=\!\!\intl_{d(x,\t)=\rho}\!\! f(x) d_\t x, \qquad (\tilde R_\rho f)(\t)=\!\!\intl_{d(x,\t)< \rho}\!\! f(x) dx, \ee
where $\t \in \cgnk$,   $d (x, \t)$ is the  Euclidean distance between $x$ and $\t$,  $d_\t x$ is the relevant measure (the precise meaning of these integrals is given  in (\ref{Rz})).
The corresponding $k$-plane transform has the form  $(R f)(\t)=(R_\rho f)(\t)|_{\rho =0}$.

Let $V_{n, n-k} \sim O(n)/O(k)$ be the Stiefel manifold
   of all $n\times (n-k)$ real matrices $v$, the columns of which
 are mutually orthogonal unit vectors.  Then each $\t \in \cgnk$ is represented as
 \be\label{hpplk} \t\equiv \t(v, t)=\{x\in \bbr^n: v^T x =t\}, \qquad v\in
 V_{n, n-k}, \quad t\in \bbr^{n-k}. \ee
  In the following, we mostly deal with the Stiefel parametrization  (\ref{hpplk}), which suits our purposes better.
 In particular, we write the $k$-plane transform $(R f)(\t)$ as
\be\label{Rz1p} \vp_v (t)\equiv(R f) (v, t)=\intl_{v^{\perp}} f(vt +
u) \,d_v u,\ee
 where integration is performed against the Euclidean  volume element $d_v u$ on the $k$-dimensional subspace $v^{\perp}$.
 In this notation,   the integrals in (\ref {Radonz})  are explicitly written  as
\be\label{Rz}
(R_\rho f) (v, t)\!=\!\!\!\intl_{\snk} \!\!\!\!\vp_v (t \!-\! \rho\theta) d_*\theta, \quad (\tilde R_\rho f) (v, t)\!=\!\!\!\intl_{|y|<1}\!\! \vp_v (t\! -\! \rho y) dy,\ee
and resemble  the mean value operators (\ref{mf}) and (\ref{hena}). As in Problem 1, it is instructive to treat these operators as the members  of the analytic family
\bea
(R^{\a}_\rho f) (v, t)&=&  \frac{\Gamma\left(  \alpha\!+\!(n-k)/2\right)  }{\pi^{(n-k)/2}\Gamma\left(
\alpha\right)  }     \intl_{|y|<1}\!\! (1\!-\!|y|^{2})^{\alpha-1} \vp_v (t\! -\! \rho y) dy\nonumber\\
\label{Rza}
&=& (\tilde M^{\a}_\rho \vp_v) (t);\eea
cf.   (\ref{hen}) with $n$ replaced by $n-k$.
Thus the main results can be obtained if we combine Theorem \ref{laaTi}  with the properties of the $k$-plane transform. For the sake of simplicity, we restrict to  $\alpha\ge 0$.
 The following statement is well known; see, e.g., \cite [Corollary 2.4]{Ru04}.

\begin{lemma}\label {zgrn} If $f\in L^p (\rn)$, $1\le p<n/k$, then
 $(R f)(v, t)$ is finite for almost all $v$ and $t$. If  $p \ge
n/k$ and
$$ f(x)=(2+|x|)^{-n/p} (\log (2+|x|))^{-1} \quad (\in
L^p(\bbr^n)),$$
 then $(R f)(v, t) \equiv \infty$.
\end {lemma}

According to this lemma and the equality (\ref{Rza}), our consideration will be necessarily restricted to the values $1\le p<n/k$.

The most complete result is obtained for $k=n-1$.

\begin{theorem}\label {lzzt} Let $\alpha\ge 0$,  $n\ge 2$, $k=n-1$.
 Then $R_\rho^{\a}$ is  injective on  $L^p (\rn)$   for any fixed $\rho >0$ and any $1 \le p < n/(n-1)$.
\end{theorem}

The case $k<n-1$ is  less complete.
\begin{theorem}\label {lzzt2} Let $\alpha\ge 0$, $f \in L^p (\rn)$, $1 \le p < n/k$,  $1\le k\le n-2$.

\vskip 0.2 truecm

\noindent {\rm (i)} If $\;1\le p\le 2n/(n+k-1)$, then $R_\rho^{\a}$ is  injective on  $L^p (\rn)$   for any fixed $\rho >0$.

\vskip 0.2 truecm

\noindent {\rm (ii)} If $\;2n/(n-1) <p <n/k$ and $ J_{(n-k)/2 +\a -1} (\rho)=0$,
   then $R_\rho^{\a}$ is not injective on  $L^p (\rn)$.

\vskip 0.2 truecm

\noindent {\rm (iii)} The cases
\bea n&=&3:  \quad k=1, \qquad 2<p < 3;\nonumber\\
 n&=&4: \quad  k=1, \qquad   2<p \le 8/3;\nonumber\\
 n&=&4: \quad  k=2, \qquad   8/5<p < 2;\nonumber\\
 n&\ge& 5: \quad k< (n-1)/2, \qquad 2n/(n+k-1)<p\le 2n/(n-1). \nonumber\eea
remain unresolved. \footnote {Note that $k< (n-1)/2$ in the last line is equivalent to $2n/(n-1)<n/k$, where $n/k$ is the upper bound for $p$.}

\vskip 0.2 truecm

\noindent {\rm (iv)} For all $\alpha\ge 0$, the function $f \in L^p (\rn)$, $1 \le p < n/k$, can be uniquely reconstructed from $R_\rho^{\a}f $  if the latter is known for any two  fixed values $\rho=\rho_1$ and $\rho=\rho_2$, provided that $\rho_1/\rho_2$ is not a quotient of zeros of the Bessel function $ J_{(n-k)/2 +\a -1}$.

\end{theorem}

\subsection {Examples} In the following, the  terminology for  different particular cases of $R_\rho^{\a}$ is provisional.
The parameter  $\rho >0$ is  fixed. We recall that $p$ is necessarily less than $n/k$; cf. Lemma \ref{zgrn}.

\vskip 0.2 truecm

$\bullet$ {\it  Strips  in $\bbr^2$} ($k=1$, $\a=1$): \\
${}\qquad $ $R_\rho^{1}$ is injective on $L^p (\bbr^2)$ for all $1\le p <2$.

\vskip 0.2 truecm

$\bullet$ {\it Slabs  in $\bbr^3$} ($k=2$, $\a=1$): \\
${}\qquad $ $R_\rho^{1}$ is injective on $L^p (\bbr^3)$ for all $1\le p <3/2$.

\vskip 0.2 truecm

$\bullet$ {\it Pipes  in $\bbr^3$} ($k=1$, $\a=0$): \\
${}\qquad $ $R_\rho^{0}$ is  injective on $L^p (\bbr^3)$ for all $1\le p \le 2$; \\
${}\qquad $  {\bf unresolved} for $2<p < 3$.

\vskip 0.2 truecm

$\bullet$ {\it Solid tubes  in $\bbr^3$} ($k=1$, $\a=1$);\\
${}\qquad $  $R_\rho^{1}$ is injective on $L^p (\bbr^3)$ for all $1\le p \le 2$;\\
${}\qquad $   {\bf unresolved} for $2<p < 3$.

\vskip 0.2 truecm

$\bullet$ {\it Pipes in $\bbr^4$} ($k=1$, $\a=0$): \\
${}\qquad $ $R_\rho^{0}$ is  injective on $L^p (\bbr^4)$ for all $1\le p \le 2$; \\
${}\qquad $  {\bf unresolved} for $2<p \le 8/3$; \\
${}\qquad $ non-injective if $8/3 < p<  4$  and  $J_{1/2} (\rho)=0$.

\vskip 0.2 truecm

$\bullet$ {\it Solid tubes  in $\bbr^4$} ($k=1$, $\a=1$): \\
${}\qquad $ $R_\rho^{1}$ is injective on $L^p (\bbr^4)$ for all $1\le p \le 2$; \\
${}\qquad $   {\bf unresolved} for $2<p \le 8/3$; \\
${}\qquad $ non-injective if $8/3 < p<  4$ and  $J_{3/2} (\rho)=0$.

An interested reader may continue this list of  examples by considering, for instance, the pairs of parallel
lines in $\bbr^2$ ($k=1$, $\a=0$) or the pairs of parallel planes in $\bbr^3$ ($k=2$, $\a=0$).

\section{ Preliminaries}

\subsection{ Notation}
${}$

\noindent In the following, $\bbr^{n}$  is the real $n$-dimensional Euclidean space;
$S^{n-1} \subset \rn$ is the $(n-1)$-dimensional unit sphere with the  normalized surface area measure $d_*\theta$.
 The points in $\bbr^{n}$ are identified with the corresponding column vectors.
  The notation $\langle f,g\rangle$ for  functions $f$ and $g$ is used for the integral of the product of these functions.
 We keep the same notation  when $f$ is a distribution and $g$ is a test function.

Given an integer $k$,  $1\le k \le n-1$,   let
$V_{n, n-k}$ be the Stiefel manifold  of orthonormal  $(n-k)$-frames in $\bbr^{n}$. Every element $v\in V_{n, n-k}$ is an  $n\times (n-k)$  matrix  satisfying
$v^T v=I_{n-k}$, where $v^T$ is the transpose of $v$ and $I_{n-k}$ is the identity $(n-k)\times (n-k)$ matrix;
 $v^\perp$ is the
 $k$-dimensional linear subspace of $\bbr^{n}$ orthogonal to $v$.

The Fourier transform  of a function
$f \in  L^1 (\bbr^n)$ is defined by
\be \label{ft}  \hat f (\xi) = \intl_{\bbr^{ n}} f(x) \,e^{ i x \cdot \xi} \,dx, \qquad \xi\in \rn,
\ee
where $x \cdot \xi = x_1 \xi_1 + \ldots + x_n\xi_n$.  We denote by $S(\bbr^n)$
the Schwartz space of $C^\infty$-functions
which are rapidly
decreasing together with their derivatives of all orders.
The space of tempered distributions, which is dual to $S(\bbr^n)$, is denoted by $S'(\bbr^n)$.  The Fourier transform of a distribution $f\in S'(\bbr^n)$ is  a
distribution $\hat f\in S'(\bbr^n)$ defined by
\be\label{ftrd12}
\lang \hat f,g \rang =\lang f,\hat g \rang,\qquad g\in S(\bbr^n). \ee
The inverse Fourier transform of a function (or distribution) $f$ is denoted by $\check f$.
 If $f\in L^1(\bbr^n)$, then $\check f(x) = (2\pi)^{-n} \hat f (-x)$.

 A function (or distribution) $f$ on $\rn$ is said to be radial if $f=f\circ \gam$ for every orthogonal transformation  $\gam \in O (n)$. The case $n=1$ corresponds to even functions or distributions.

\subsection{Bessel functions}

For $r>0$, the  Bessel function of the first kind \cite{Er} has the form
\[
J_\nu (z)=\sum\limits_{m=0}^\infty \frac{(-1)^m \, (r/2)^{2m +\nu}}{m!\, \Gam (m+\nu +1)}.
\]
We normalize $J_\nu (r)$ by setting
\be\label{nrfl}
j_\nu (r)= \Gam (\nu +1) (r/2)^{-\nu}J_\nu (r).\ee
Then
\be\label{nrfl2}
\lim\limits_{r \to 0}j_\nu (r) =1; \qquad j_\nu (r)=O(r^{-\nu-1/2})\quad \text{\rm as }\quad r\to \infty.\ee

Denote
\be\label{nrwz}
\psi (x) =\intl_{\sn} e^{ i x \cdot \theta} \,d_*\theta, \qquad   x\in \rn.\ee
Then
\be\label{nrwz1}
\psi (x)= j_{n/2 -1}(|x|).\ee
By (\ref{nrfl2}),
\be\label{nuilp}
\psi  \in L^p (\rn), \qquad p\,  \left \{ \begin{array} {ll}  > \,2n/(n-1)& \mbox{if $n>1$,}\\
=\infty & \mbox{if $n=1$.}\\
\end{array}
\right.\ee

We will be dealing with the Erd\'elyi-Kober type fractional integrals \cite [Section 2.6.2]{Ru15} defined by
\be
\label{eci} (I^\a_{-, 2} f)(r)=\frac{2}{\Gam
(\a)}\intl_r^\infty (s^2 - r^2)^{\a -1} f (s)  \, s\, ds, \qquad  \a >0.\ee
By \cite[formula 2.12.4(17)]{PBM},
\be
\label{ecis} (I^\a_{-, 2} j_\nu)(r)=\frac{2^{2\a} \Gam (\nu +1)}{\Gam (\nu -\a +1)}\, j_{\nu-\a} (r), \qquad 2\a -\nu <\frac{3}{2}.\ee

\subsection{On  the Wiener algebra}
Let
\[\W_0 (\rn) =\{f: f(\xi) = \hat g (\xi), \;g \in  L^1 (\rn) \}\]
 be the  Wiener algebra (or Wiener's ring) of the Fourier transforms of $L^1$-functions.

\begin{lemma}\label{ring} {\rm \cite [Lemma 1.3]{Sa77}, \cite [Lemma 1.22] {Sa}}  Let $f\in L^1 (\rn)$. If $f$ has  mixed derivatives $\partial^j f$ belonging to $ L^p (\rn)$ for
some $1<p\le 2$ and  all multi-indices $j \in \{0,1\}^n$, $j\neq 0$,  then $f\in  \W_0 (\rn)$.
\end{lemma}

\begin{corollary} \label {corr} {\rm \cite [Lemma 1.3]{Ru89}}
Let $m(\xi)=\mu (|\xi|)\in  L^1 (\rn)$,
\[m_k(\xi)= \mu_k (|\xi|), \qquad \mu_k (r)= r^k \left (\frac {1}{r}\, \frac {d}{dr}\right )^k \mu (r).\]
If $m_k\in  L^p (\rn)$ for some $1<p\le 2$ and all $k=1,2,\ldots, n$, then  $m \in  \W_0 (\rn)$.
\end{corollary}
This corollary  is especially useful when $m(\xi)$ is a function of  $|\xi|^2$.


\section{Proof of Theorem \ref{laaTi}}

The proof of Theorem \ref{laaTi} reduces to the study of injectivity of the operator $M_\rho^{\a}=M_t^{\a}|_{t=\rho}$ on $L^p$-functions. We recall
that $M_\rho^{\a}f =m_\rho^{\a} \ast f$ (in the $S'$-sense), where $m_\rho^{\a}$ is defined by (\ref{mmb}) with $t=\rho$.
 Following \cite{B2}, we  write  $M_\rho^{\a}f$ in the Fourier terms as
\be\label{oqzu}
[M_\rho^{\a}f]^{\wedge}(\xi)= j_\nu (\rho|\xi|)\, \hat f (\xi), \qquad  \nu=n/2 +\a -1, \quad f\in S(\rn), \ee
 where $j_\nu$ is the normalized Bessel function (\ref{nrfl}).

 \begin{lemma} \label {nrw4}  Let $f\in  L^p (\rn)$ where  $p> 2n/(n-1)$ if $n>1$ and $p<\infty$ if $n=1$. If $\supp \,\hat f$ is carried by the sphere $\vert \xi \vert =\rho$, $\rho >0$, then  $f=0$.
  \end{lemma}
 This lemma is a particular case of Theorem 1 from \cite{AN}.  A similar statement for radial functions $f$ and $n\ge 2$ follows from Lemmas 2.1 and 2.2 in \cite [p. 309]{RS}.
 Note that if $n=1$, then the sphere $S_\rho$ consists  of two points $\pm \rho$.

\begin{proposition}\label {laa} Let $\alpha\ge (1-n)/2$.

\vskip 0.2 truecm

\noindent {\rm (i)} Suppose that  $1\le p\le  2n/(n-1)$ if $n>1$ and $1\le p< \infty$ if $n=1$. Then the operator   $M_\rho^{\a}$ is  injective on  $L^p (\rn)$ for  any fixed $\rho>0$.

\vskip 0.2 truecm

\noindent {\rm (ii)} Let  $p> 2n/(n-1)$ if $n>1$, or $p=\infty$ if $n=1$. If  $J_{n/2 +\a -1} (\rho)=0$,   then $M_\rho^{\a}$ is not injective on  $L^p (\rn)$.
\end{proposition}
\begin{proof} {\rm (i)}  We proceed as in the proof of Proposition 3.1 in \cite{RS} with some changes. In particular, we prefer to use  Lemma \ref{nrw4} 
 for our treatment.  Let  $M_\rho^{\a} f =0$ for some $f\in L^p (\rn)$.
 Our aim is to show that $f=0$ a.e.  Replacing $f$ by
the convolution with the corresponding approximate identity, we can assume that $f$ is smooth. Suppose that $f \not\equiv 0$,
that is, $f(x_0) \neq 0$ for some  point $x_0$. Let us  show that this assumption yields a contradiction.

By the translation invariance
of $M_\rho^{\a}$, we can assume $x_0=0$, that is, $f(0) \neq 0$. Then the radial function $f_0 (x)=\int_{O (n)} f(\gam x) d\gam$ also satisfies $f_0(0) \neq 0$,
 This function is  smooth,  belongs to $L^p (\rn)$,  and obeys $M_\rho^{\a} f_0 =0$.
Thus, in the following,  we may assume that $f$ is smooth, radial, nontrivial,  belongs to $L^p (\rn)$, and $M_\rho^{\a} f =0$. It follows
that $\hat f$ is a nontrivial radial  distributions, and therefore $\supp \,  \hat f \neq \emptyset$. Then there is a point  $\xi_0 \neq 0$ such that $\xi_0 \in \supp \,  \hat f$.
Since  $\hat f$ is radial, it follows that the whole sphere $S_\rho$ of radius $\rho= \vert\xi_0\vert $ is contained in $\supp \,  \hat f$.

By (\ref{oqzu}), the equality $M_\rho^{\a} f =0$ implies
\be\label {nazwa}
j_\nu (\rho \vert \xi\vert)\, \hat f (\xi)=0, \qquad \nu=n/2 +\a -1,\ee
 in the $S'$-sense.
If $\{z_1, z_2, \ldots\}$ is the set of all zeros of the Bessel function $J_\nu$ and $S_i=\{\xi \in \rn: \vert \xi\vert =z_i/\rho\}$, then (\ref{nazwa}) means that
$\supp \,  \hat f \subset \,\bigcup_{i}  S_i $.
Hence $S_\rho =S_{i^*}$ for some $i=i^*$.

Given a sufficiently small $\e >0$, we choose a smooth radial function $\psi$ such that
$\psi (\xi) \equiv 1 $ if $\vert \xi\vert \in  [\rho - \e, \rho +\e] $ and $\psi (\xi) \equiv 0 $ if  $\vert \xi\vert \notin [\rho - 2\e, \rho + 2\e]$.  If $\e$ is small enough, then $S_\rho =S_{i^*}$ is the only sphere in the union $\bigcup_{i}  S_i$, which is contained in $\supp \,  \hat f$. It follows that
$\supp \,  \psi\hat f =S_\rho$. However,
$\psi\hat f= (\check \psi \ast f)^{\wedge}$, where $\check \psi \ast f \in L^p (\rn)$. Then, by Lemma \ref{nrw4}, $\check \psi \ast f \equiv 0$. This gives a contradiction because
$\check \psi \ast f$ is nontrivial.

\vskip 0.2 truecm

  {\rm (ii)} To prove the second statement, which is
well known  for $\a=0$,  we take
 the function $\psi (x)$ from  (\ref{nrwz}) as a  counter-example. Specifically,
 let us show that
\be\label{nrwz3}
(M_\rho^{\a} \psi) (x)= j_\nu (\rho) \psi (x), \qquad \nu=n/2 +\a -1.\ee
In the case $Re \, \a > 0$, changing the order of integration, we obtain
\bea
(M_\rho^{\a} \psi) (x)&=&
\frac{\Gamma\left(  \alpha\!+\!n/2\right)  }{\pi^{n/2}\Gamma\left(
\alpha\right)  }\intl_{|y|<1}\!\!(1\!-\!|y|^{2})^{\alpha-1} dy \intl_{\sn} e^{i (x- \rho y)\cdot \theta} d_*\theta\nonumber\\
&=& \intl_{\sn} e^{i  x \cdot \theta} \Bigg [ \frac{\Gamma\left(  \alpha\!+\!n/2\right)  }{\pi^{n/2}\Gamma\left(
\alpha\right)  }\intl_{|y|<1}\!\!(1\!-\!|y|^{2})^{\alpha-1} e^{-i  \rho y\cdot \theta}dy \Bigg ] d_*\theta \nonumber\\
 &=&    j_\nu (\rho) \psi (x).\eea
 The result for all $Re\, \alpha\ge (1-n)/2$ then follows by analytic continuation (in the $S'$-sense).
\end{proof}

The next statement mimics the two-radius theorem from  \cite [Theorem 2.1] {T}, corresponding to $\a=0$.

\begin{proposition}\label {laaT} Let $\alpha\ge (1-n)/2$, $n \ge 1$,  $f\in L^p (\rn)$,  $1\le p \le \infty$. If  $M_{\rho_1}^{\a} f=0$ and $M_{\rho_2}^{\a} f=0$, then $f=0$ a.e.,
provided that $\rho_1/\rho_2$ is not a quotient of zeros of the Bessel function   $J_{\nu}$, $\nu=n/2 +\a -1$.
\end{proposition}
\begin{proof} We  follow  \cite [Theorem 2.1] {T} with some changes, but keeping the same  idea.
 It suffices to show that $\int_{\rn} f(x) g(x) dx =0$ for all $g \in S(\rn)$. Let us replace $g$  by its approximation, for instance, by the Gauss-Weierstrass integral $(W_\e g)(x)$ \cite {SW}, for which
$(W_\e g)^{\wedge}(\xi)= e^{-\e |\xi|^2} \,\hat g(\xi)$.
Denote $M_{1,2} = M_{\rho_1}^{\a} +i M_{\rho_2}^{\a}$, so that
\[ M_{1,2} f = m_{1,2} \ast f, \qquad \hat m_{1,2} (\xi)=j_\nu (\rho_1 |\xi|) + i j_\nu (\rho_2 |\xi|).\]
 Since $\rho_1/\rho_2$ is not a quotient of zeros of  $J_{\nu}$, the function
\[(j_\nu (\rho_1 |\xi|) + i j_\nu (\rho_2 |\xi|))^{-1}\]
is bounded. Moreover, the function
 \[m_\e (\xi) = \frac{e^{-\e |\xi|^2}}{j_\nu (\rho_1 |\xi|) + i j_\nu (\rho_2 |\xi|)}\]
 satisfies Corollary \ref{corr} (we leave simple calculations to the reader),  and therefore $\check m_\e \in L^1 (\rn)$ for all $\e>0$.
  We have
 \[(W_\e g)^{\wedge}(\xi)= e^{-\e |\xi|^2} \hat g(\xi)= m_\e (\xi) \,\hat m_{1,2}  (\xi)\, \hat g(\xi).\]
  Hence, by the Plancherel formula,
 \bea
 \lang f, \overline {W_\e g}\,\rang &=& (2\pi)^{-n} \lang \hat f, \overline { m_\e  \hat m_{1,2} \hat g  }\, \rang\nonumber\\
 &=& (2\pi)^{-n} \lang \hat f \hat m_{1,2}, \overline { m_\e  \hat g  } \,\rang = \lang  M_{1,2} f, \check m_\e \ast g\rang =0\nonumber\eea
  because, by the assumption, $M_{1,2} f=0$. This gives the result.
\end {proof}

\begin{remark} \label {kasz} Some comments are in order, related to the difference between our proof of  Proposition \ref {laaT}  and the proof of Theorem 2.1 in \cite {T}.
We consider a more general class of oscillatory integrals and, unlike \cite {T}, don't invoke the machinery of  Strichartz' spectral decomposition in eigenfunctions of the Laplacian \cite{Str89}.
 The Bochner-Riesz means in \cite {T} are replaced by the Gauss-Weierstrass integrals.
 The use of the Mikhlin-H{\"o}rmander multiplier theorem for $1<p<\infty$ is substituted by Corollary \ref{corr} for Wiener's algebra.
 Although this technique is less powerful and conceptually simpler, it works well in our case and includes all $1\le p \le \infty$.
\end{remark}

Propositions \ref{laa} and \ref{laaT} yield Theorem \ref{laaTi}  in Introduction.

\section{Proof of Theorems \ref{lzzt} and \ref{lzzt2}}

\subsection{Auxiliary facts}

Let $1\le k\le n-1$. We recall that the Radon-John $k$-plane transform  takes functions on $\bbr^n$ to functions on the affine Grassmannian  $\cgnk$  by the formula
\be\label{John} (Rf)(\t)=\intl_{\t} f(x)\, d_\t x, \qquad \t \in \cgnk,\ee
  $ d_{\t} x$ being the Euclidean measure on $\t$. If $k\!=\!1$, (\ref{John}) is known as the $X$-ray transform.

 Denote $Z_{n,k}=  V_{n, n-k} \times \bbr^{n-k}$.  We will be working with the Stiefel parametrization  (\ref{Rz1p}) of the planes $\t=\t(v,t) \in \cgnk $,  $(v,t) \in Z_{n,k}$.
Consider the mixed  norm space
\bea\label{ttyyfd}
&& L^{q, r} (Z_{n,k})= \Bigg\{ \vp (v, t): \\
&& \| \vp
\|_{q, r}  \! = \! \Bigg(\, \intl_{V_{n, n-k}}\! \Bigg [\,\intl_{\rnk}
|\vp(v, t)|^r dt\Bigg]^{q/r}  \! \! \! \!dv\Bigg)^{1/q}
 \!\! \!\!< \!\infty \Bigg \}; \nonumber \eea
$1\le q, r\le \infty$. Our main objective is the norm inequality
\be\label {ozut} \| Rf\|_{q, r} \le c_{p, q, r}  \| f \|_p\, ; \ee
 see  Oberlin and Stein \cite{OS},  Strichartz \cite{Str1},  Drury \cite {Dr89},  Christ \cite{Chr84}. It is convenient to consider the cases $k=n-1$ and $k<n-1$ separately.
  The following statements hold.

\begin{theorem} \label{Oberlin} \cite {OS, Str1} Let $n\ge 2$,  $k=n-1$, $  1/p + 1/p' = 1$.
 Then
 (\ref{ozut})
holds if and only if \[1\le p < n/(n-1), \qquad  q \le
p^{\prime},\qquad r^{-1} = np^{-1} - n+1.\]
\end{theorem}

\begin{theorem} \label{CDru} \cite [Theorems A,B]{Chr84}, \cite {Dr89}   Let $1\le k \le n-2$, $1\le p< n/k$,
\be\label {ozut1}
\frac{n}{p}- \frac{n-k}{r}=k, \qquad 1\le q \le (n-k)p'.\ee
Then  (\ref{ozut}) is true in the following cases:
\bea &&{\rm (i)}\quad  \text{if} \quad  1\le p \le (n+1)/(k+1)\quad \text{for all $\;k$};\nonumber\\
&& {\rm (ii)} \quad \text{in the full range  $1\le p< n/k$ } \quad  \text{for $\;k>n/2$}.\nonumber\eea
\end{theorem}

 The relations  (\ref{ozut1}) are necessary for  (\ref{ozut}). Examination of (i) and (ii) shows that the case
\[ k<n/2, \qquad (n+1)/(k+1) <p <n/k,\]
 is not
covered by Theorem \ref{CDru}. To the best of my knowledge, this case   remains unresolved.

We recall  that the $k$-plane transform  is injective on $L^p(\rn)$, $1\le p< n/k$. Inversion formulas for $Rf$ can be found, e.g., in \cite{H11, Ru04}.

The radial case deserves special mentioning.

\begin{lemma}\label {zhujun} {\rm (cf. \cite[Lemma 2.1]{Ru04})}  If   $f(x) =f_0(|x|)$, then $(R f)(v, t) = F_0(|t|)$, where
\be\label{ppaawsdz}
F_0(s)=\frac{2\pi^{k/2}}{\Gam (k/2)}\intl_s^\infty f_0(r)
(r^2 -s^2)^{k/2 -1} r dr=\pi^{k/2} \,(I^{k/2}_{-,2} f_0)(s), \ee
 provided that this  integral  exists in the Lebesgue sense.
\end{lemma}

\subsection{Proof of Theorem \ref {lzzt}}

Let us consider the  operator family $R^{\a}_\rho$, $\a\ge 0$, including the shifted $k$-plane transform $R_\rho$ ($\a=0$) and its modification  $\tilde R_\rho$ ($\a=1$); see (\ref{Rz}), (\ref{Rza}).
 By (\ref{Rza}),
\be\label {iir} (R^{\a}_\rho f) (v, t)= (\tilde M^{\a}_\rho \vp_v) (t), \qquad \vp_v (t)=(R f) (v, t),\ee
where $v\in \sn$, $t\in (-\infty, \infty)$. By Theorem \ref{Oberlin},
\[
\vp_v \in L^r (-\infty, \infty), \qquad r^{-1} = np^{-1} - n+1,\]
 for almost all $v\in \sn$. If $R_\rho^{\a} f =0$, then $\tilde M^{\a}_\rho \vp_v=0$. Then
 Proposition \ref{laa}(i) (with $p$ replaced by $r$) yields $\vp_v (t)\equiv (R f) (v, t)=0$ for any
 fixed $\rho >0$ and almost all $(v, t) \in \sn \times \bbr$. Hence $f=0$ a.e., owing to injectivity of the Radon transform.
 \hfill $\Box$

\subsection{Proof of Theorem \ref {lzzt2}}

\noindent  (i)  We proceed as in (\ref{iir}), where, by Theorem \ref{CDru},
\[
\vp_v \in L^r (\rnk), \qquad r=\frac{n-k}{n/p -k},\]
 for almost all $v\in V_{n, n-k}$. If $R^{\a}_\rho f=0$, then  $\tilde M^{\a}_\rho \vp_v=0$. Hence
  Proposition  \ref{laa}(i) (with $p$ replaced by $r$) yields $\vp_v (t)\equiv (R f) (v, t)=0$, provided that
\[
r=\frac{n-k}{n/p -k} \le \frac{2(n-k)}{n-k-1}.\]
The latter is equivalent to
\be\label {iirc} 1\le p\le 2n/(n+k-1).\ee
 Combining this inequality with  Theorem \ref{CDru}, we arrive at the two options:
\bea  &&{\rm (A)}\quad   1\le p \le \min \left \{\frac{n+ 1}{k+1};\; \frac{2n}{n+k-1}\right\},\quad 1\le k \le n-2;\nonumber\\
&& {\rm (B)}\quad 1\le p\le \frac{2n}{n+k-1}, \quad n/2<k \le n-2.  \quad \nonumber\eea
A simple calculation shows that (A) splits in  two subcases. Specifically, because
\[
\min  \{ \cdots \}=  \left \{ \!\begin{array} {ll} \displaystyle{\frac{2n}{n+k-1}}& \mbox{if $\; 1\le k \le c_n$},\\
{}\\
\displaystyle{\frac{n+ 1}{k+1}} & \mbox{if $\; c_n \le k\le n-2$;}\\
\end{array}
\right.   \; c_n=  n\!-\!1\! -\! \frac{2}{n\!-\!1},\]
we have
\bea  &&({\rm A}_1) \quad   1\le p \le  2n/(n+k-1)\quad \text{\rm if} \quad 1\le k \le c_n,\nonumber\\
&& ({\rm A}_2) \quad 1\le p \le (n+ 1)/(k+1), \quad \text{\rm if} \quad    c_n \le k\le n-2.\nonumber\eea
In view of (B), further consideration depends on whether $c_n \le n/2$  or $c_n > n/2$.  The inequality $c_n \le n/2$ holds if and only if $n=3$ and corresponds to $k=1$.
 Combining $({\rm A}_1)$, $({\rm A}_2)$, and (B), and taking into account that $R$ is injective on  functions $f\in L^p (\rn)$, $1 \le p < n/k$, we obtain
  the statement (i) of the theorem.


(ii) Consider the situation when the injectivity may not hold. It suffices to show that if
$f(x)=j_{n/2 -1} (|x|)$
 (cf. (\ref {nrwz1}), (\ref{nuilp})), then, for all $v \in  V_{n, n-k}$,
\be\label {lqrew} (R_\rho^{\a} f)(v,t) = c\, j_{(n-k)/2 +\a-1} (\rho)\, j_{(n-k)/2 -1} (|t|), \qquad c=\const.\ee
By (\ref{Rz1p}) and Lemma \ref{zhujun}, we have
$\vp_v (t)=\pi^{k/2} \,(I^{k/2}_{-,2} j_{n/2 -1})(|t|)$.
Hence, by (\ref{ecis}) with $\a=k/2$ and $ k < (n-1)/2$,
\[ \vp_v (t)\!=\!c\, j_{(n-k)/2 -1} (|t|), \qquad  c= \frac{2^{k} \pi^{k/2}\,\Gam (n/2)}{\Gam ((n\!-\!k)/2)}.\]
 Then (\ref{lqrew}) follows from (\ref {Rza}) and  (\ref{nrwz3}) with  $n$  replaced   by $n-k$.

The statement (iii) follows from routine examination of parts (i) and (ii). The statement (iv) follows from (\ref{iir}) and  Proposition \ref{laaT} with $n$ replaced by $n-k$. \hfill $\Box$

\end{document}